\documentclass[a4paper,11pt]{amsart}
\usepackage[T1]{fontenc}
\usepackage[english]{babel}
\usepackage[cp1252]{inputenc}
\usepackage{amsthm}
\usepackage{amsmath}
\usepackage{amsfonts}
\usepackage{amssymb}
\usepackage{hyperref}
\usepackage{color}
\usepackage{caption}
\usepackage{indentfirst}
\usepackage{amssymb}
\usepackage{eufrak}
\usepackage{mathrsfs}
\usepackage{xypic}
\usepackage{graphicx, subfigure}

\newcommand{\Cr}{\mathscr C}
\newcommand{\V}{\textrm{Vert}}

\newcommand{\R}{\mathbb R}
\newcommand{\C}{\mathbb C}

\newcommand{\Arg}{\mathrm{Arg}\,}

\newcommand{\Log}{\mathrm{Log}\,}
\newcommand{\A}{\mathscr A}

\newcommand{\Li}{\mathscr{L}^\infty}

\newcommand\restr[2]{{% we make the whole thing an ordinary symbol
  \left.\kern-\nulldelimiterspace % automatically resize the bar with \right
  #1 % the function
  \right|_{#2} % this is the delimiter
  }}

\newtheorem*{mtheorem*}{Main Theorem}
\newtheorem{theorem}{Theorem}[section]
\newtheorem{definition}{Definition}[section]
\newtheorem{proposition}{Proposition}[section]
\newtheorem{corollary}{Corollary}[section]
\newtheorem{lemma}{Lemma}[section]
\begin{document}
\title[A tropical characterization of complex analytic varieties]{A tropical characterization of complex analytic varieties to be algebraic}
\author{Farid Madani, Lamine Nisse and Mounir Nisse}
\date{}

\address{NWF I-Mathematik, Universit\"at Regensburg, 93040 Regensburg, Germany.}
%\curraddr{}
\email{\href{mailto:Farid.Madani@mathematik.uni-regensburg.de}{Farid.Madani@mathematik.uni-regensburg.de}}

\address{Laboratory of Applied Mathematics, University of Badji Mokhtar - Annaba, Algeria.}
\email{\href{mailto:lamine.nisse@univ-annaba.dz}{lamine.nisse@univ-annaba.dz}}
\address{School of Mathematics KIAS, 87 Hoegiro Dongdaemun-gu, Seoul 130-722, South Korea.}
\email{\href{mailto:mounir.nisse@gmail.com}{mounir.nisse@gmail.com}}

%\urladdr{www.math.tamu.edu/\~{}nisse}
\thanks{Research of the third author is partially supported by NSF grant DMS-1001615, and Max Planck Institute for Mathematics, Bonn, Germany.}
%%%%%%%%%%%%%%%%%%%%%%%%%%%%%%%%%%%%%%%%%%%%%%%%%%%%%%%%%%%%%%%%%%%%%%%%%%%%

\maketitle

\begin{abstract}  In this paper we study a  $k$-dimensional analytic subvariety  of  the complex algebraic torus. We show that if its logarithmic limit set is a finite rational  $(k-1)$-dimensional spherical polyhedron, then each irreducible component of the variety is algebraic. This gives a converse of a theorem of Bieri and Groves and generalizes a result proven in \cite{MN2-11}.  More precisely, if the dimension of the ambient space is at least twice of the dimension of the generic analytic subvariety, then these properties are  equivalent to the volume of the amoeba of the subvariety being  finite. 
\end{abstract}

%\setcounter{tocdepth}{1} \tableofcontents

%%%%%%%%%%%%%%%%%%%%%%%%%%%%%
\section{Introduction}
%%%%%%%%%%%%

Amoebas of complex varieties with their cousin  coamoebas  play a major role as a link between  complex algebraic geometry and tropical geometry. Moreover, they are used in several areas of mathematics, in real algebraic geometry, mirror symmetry, algebraic statistics, complex analysis (see \cite{MS-09}, \cite{M2-04},
 \cite{FPT-00}, \cite{NS-12}, \cite{PR-04}, and \cite{PS-04}). We show in this paper that the logarithmic limit sets and the phase limit sets play a role as crucial  as the role played by their relatives. Indeed, their role is a link between complex algebraic geometry and phase tropical geometry. More precisely, this is a quadruplet  (logarithmic limit set, amoeba, coamoeba, phase limit set) and we can not dissociate one of these objects from the others.

We do believe that these objects are not yet fully exploited, and  they contain more information about our original object which is the complex variety. This information can be  apparently of different  nature  e.g., geometric, algebraic,  topological, combinatorial. But they are often equivalent. We  prove in this paper the equivalence  between  some properties of algebraic nature on one hand  and  some properties of  combinatorial and topological  nature, on the other hand. Another equivalence between  geometric  nature properties  and of algebraic nature was proven in \cite{MN2-11}.  More precisely, we show that $k$-dimensional irreducible analytic subvarieties  of the complex torus are algebraic, if their logarithmic limit sets are finite rational  $(k-1)$-dimensional complex polyhedrons.  In addition, if the dimension of the ambient space is at least the double of the dimension of the varieties, then,  these properties are equivalent to the fact that the volume of the amoebas of the varieties are finite, which is a completely geometric property. The last  particular case is proven in \cite{MN2-11}.

In \cite{B-71}, Bergman introduced the notion of logarithmic limit set of a subvariety of the algebraic 
torus as the set of limiting directions of points in its amoeba. 
In  \cite{BG-84}, Bieri and Groves proved the following theorem, conjectured by Bergman. 
\begin{theorem}[Bergman, Bieri--Groves]\label{BBG} The logarithmic 
limit set $\mathscr{L}^{\infty}(V)$ of an algebraic variety $V$ 
in $(\mathbb{C}^*)^n$ is a finite union of rational spherical 
polyhedrons. The maximal dimension of a polyhedron
 $P$ in this union is such that
$\dim_{\mathbb{R}}P = \dim_{\mathbb{R}}\mathscr{L}^{\infty}(V) = \dim_{\mathbb{C}}V - 1$.
\end{theorem}

The aim of this paper  is to prove the converse of Theorem \ref{BBG}. 
\begin{mtheorem*}
Let $V$ be a $k$-dimensional irreducible analytic  subvariety   of 
the complex algebraic torus $(\mathbb{C}^*)^n$ 
and $\mathscr{A}(V)$ be its amoeba.  
Let $\mathscr{L}^{\infty}(V)$  be the logarithmic limit set of $V$. 
Assume that $\mathscr{L}^{\infty}(V)$ is a  finite rational spherical polyhedron of dimension $k-1$. Then $V$
is algebraic.
\end{mtheorem*}

\vspace{0.2cm}

Chow's theorem asserts that any analytic subvariety of the projective space is algebraic. Bieri-Grove's theorem and the main theorem above give a necessary and sufficient condition for a subvariety in the complex algebraic torus to be algebraic.

The key ingredients in our proof are on one side, the topology 
and the combinatorial structure of the logarithmic limit set of our  
variety and on the other side,  their link with the geometry of 
the amoeba.

\vspace{0.2cm}

The paper is organized as follows: In Section 2, we recall some basic definitions and  introduce our  notation.  
In Section 3,  we prove the main theorem and give some consequences.

\vspace{0.2cm}
  
\paragraph{\bf Acknowledgments} 
The third author thanks Professor  Bernd Ammann for the invitation at the University of Regensburg and Max Planck Institute for Mathematics in Bonn and Korean Institute for Advanced Study for their hospitality.

%%%%%%%%%%%%%%%%%%%%%%%%%%%%%
\section{Preliminaries}\label{prelimin}
%%%%%%%%%%%%

\vspace{0.2cm}

A subset $Q$ of $\mathbb{R}^n$  is said to be a rational convex polyhedron, if $P$ can be written as the intersection
$$
Q = H_1\cap H_2\cap\cdots\cap H_l
$$
of a finite number of closed affine half spaces in $\mathbb{R}^n$,  where each $H_j$ can be defined in terms of inequalities  of the form $\sum_{i=1}^nq_{ij}x_i\leq a_j$, with rational coefficients $q_{ij}$  and real numbers $a_j$. A subset  $P\subset \mathbb{R}^n$ is said to be a {\em rational complex polyhedron} if it can be written as the union:
$$
P = Q_1\cup Q_2\cup\cdots\cup Q_r 
$$
of a finite number of rational convex polyhedrons. $P$ is said to be homogeneous of dimension $k$ if the dimension of $Q_j$ is equal to $k$, for each  $j=1,\ldots , r$.

%\begin{definition}
A  \emph{rational spherical polyhedron} is a finite union of closed hemispheres which can be written  in terms of  a finite number of inequalities with integral coefficients. 
%\end{definition}

\vspace{0.2cm}

Let $W$ be an analytic variety in $\mathbb{C}^n$ defined  globally by an ideal $\mathcal{I}$ of holomorphic functions  on $\mathbb{C}^n$. 
We say that a subvariety $V$  of the complex algebraic torus  $(\mathbb{C}^*)^n$ is analytic if there exists an analytic variety $W$ as above such that $V:=W\cap (\mathbb{C}^*)^n$. 

\emph{All the analytic varieties considered in this paper are defined as above}.

The \emph{amoeba}  $\A$ of $V$ is by definition (see M. Gelfand, M.M. Kapranov
 and A.V. Zelevinsky \cite{GKZ-94}) the image of $V$ under the map :
\[
\begin{array}{ccccl}
\Log&:&(\mathbb{C}^*)^n&\longrightarrow&\mathbb{R}^n\\
&&(z_1,\ldots ,z_n)&\longmapsto&(\log |z_1|,\ldots ,\log|
z_n|).
\end{array}
\]

%%%%%%%%%%%%%%%%%%%

The amoeba of a variety of codimension one  is closed and its complement components in $\mathbb{R}^n$ are convex (see \cite{FPT-00} ). In \cite{H-03}, Henriques  generalized the notion of convexity as follows:
 \begin{definition}
A subset $A\subset \mathbb{R}^n$ is called $l$-convex if for all oriented affine $(l+1)$-planes 
$L \subset \mathbb{R}^n$, the induced homomorphism $H_{l}(L\cap A) \longrightarrow H_{l}(A)$ does not
 send non-zero elements of $\tilde{H}_{l}^+(L\cap A)$ to zero, where $\tilde{H}_l(L\cap A)$ (resp. $\tilde{H}_l^+(L\cap A)$) denotes the reduced homology 
groups associated to the corresponding augmented complexes (resp. elements of $H_{l}(L\cap A)$ such that
 their image in
 $\tilde{H}_l(L\setminus { p})\sim \mathbb{Z}$ are non-negative for all $p\in L\cap A$).
\end{definition} 

When the subset $A$ is the complement of an amoeba, Henriques obtains the following result 

\begin{theorem}[Henriques \cite{H-03}]\label{henri}
Let $V\subset (\mathbb{C}^*)^n$ be a variety of codimension $r$ and $\A$ be its amoeba. Let $L$ be an $r$-plane of rational slope and $c$ be
 a non-zero $(r-1)$-cycle  in $H_{r-1}(L\setminus \mathscr{A}) $. Then the image of $c$ in $H_{r-1}(\mathbb{R}^n\setminus \mathscr{A})$ is non-zero and
 $\mathbb{R}^n\setminus \mathscr{A}$ is $(r-1)$-convex. 
  \end{theorem}

The logarithmic limit set $\mathscr L^\infty (V)$ of  an analytic subvariety $V$ of the complex algebraic torus is the boundary of the closure of $\rho(\A(V))$ in the $n-$dimensional ball $B^n$,
 where $\rho$ is the map defined by (see Bergman \cite{B-71}):
\[
\begin{array}{ccccl}
\rho&:&\mathbb{R}^n&\longrightarrow&B^n\\
&&x&\longmapsto&\rho(x) = \frac{x}{1+|x|}.
\end{array}
\]     
If $V$ is algebraic of dimension $k$, then its {\ logarithmic limit set} is a finite rational spherical polyhedron of dimension $k-1$.

The argument map is the map  defined as follows:
\[
\begin{array}{ccccl}
\Arg&:&(\mathbb{C}^*)^n&\longrightarrow&(S^1)^n\\
&&(z_1,\ldots ,z_n)&\longmapsto&(\arg (z_1),\ldots ,\arg (
z_n) ).
\end{array}
\]
where $\arg (z_j) = \frac{z_j}{|z_j|}$. The {\em coamoeba} of $V$,  denoted by $co\mathscr{A}$, is its image 
under the argument map (defined for the first time by Passare in 2004).

Sottile and the third author \cite{NS-12} defined the {\em phase limit set}  of $V$, $\mathscr{P}^{\infty}(V)$, 
as the set of accumulation points of arguments of sequences in $V$ with unbounded logarithm. If $V$ is an
 algebraic  variety of dimension $k$,  $\mathscr{P}^{\infty}(V)$  contains an arrangement of $k$-dimensional real
 subtori.\\
 
Now, we introduce the notion of ends of an analytic subvariety  of the 
complex algebraic torus. 
Let $V$ be an analytic variety in the complex algebraic torus $(\mathbb{C}^*)^n$, and let $\mathscr{L}^{\infty}(V)$ be its logarithmic limit set.
Let $\{ K_l\}_{l=0}^\infty$ be a family of compact subsets of $V$ such that $K_l\subset K_{l+1}$ for $l=0,\ldots , \infty$, and $V = \bigcup_{l=0}^\infty K_l$.
We define the set of ends of $V$ as follows:
$$
\mathscr{E}^\infty(V) := \bigcap_{l=0}^\infty \left( \overline{V}\setminus K_l\right) ,
$$
where $\overline{V}$ denotes the closure of $V$ in $\mathbb{C}^n$, and each connected component of $\mathscr{E}^\infty(V)$ is called an \emph{end} of $V$.

Let $x$  be a point in $\mathscr{L}^{\infty}(V)$, and $\mathscr{S}_x(V)$ be the set of sequences in $V$ defined as follows:
$$
\mathscr{S}_x(V) := \left\{  s = \{z_l\}_{l=0}^{\infty} \subset V \, |\,  x\in \Li(\{s\})  \right\} .
$$
%\overline{\left\{ \frac{\Log (z_l)}{|\Log (z_l)| + 1}\right\}_{l=0}^{\infty}}
Let $\mathscr{E}_x^{\infty}(V)$ be the subset of points in $\mathbb{C}^n$ defined as follows:
$$
\mathscr{E}_x^{\infty}(V) := \bigcup_{s= \{z_l\}_{l=0}^{\infty}\in \mathscr{S}_x(V)} \left( \bigcap_{N=0}^{\infty}\{  \overline{s}\setminus  \{z_l\}_{l=0}^{N} \}\right) .
$$
A connected component of $\mathscr{E}_x^{\infty}(V)$ is called an \emph{end of $V$ corresponding to $x$}.

\section{Proof of the main theorem}

Before starting the proof, let us note that the assumption in the main theorem on the irreducibility of the algebraic subvariety is necessary. If $V$ has an infinite number of irreducible components and its logarithmic limit set is a  finite rational spherical polyhedron of dimension $k-1$, then in general it is not possible to conclude weather $V$ is algebraic or not. 
For example, the logarithmic limit set of the  plane analytic curve  $\Cr\subset (C^*)^2$, with defining function $f(z_1,z_2)=\sin(\pi z_1 z_2)$, is $\Li(\Cr)=\{\pm(1,-1)\}$. This is a rational spherical polyhedron of dimension zero, but the curve $\Cr$ is not algebraic.

\begin{proof}[Proof of main theorem]
From now on, we assume that  $V\subset (\C^*)^n$ is a  $k$-dimensional analytic variety, such that its logarithmic limit set is a finite rational spherical polyhedron of dimension $k-1$.
 Moreover, we assume that the ideal $ \mathcal I(V)$ is generated by a set of entire functions $\{f_1,\ldots,f_q\}$, where each entire function $f_j$ can not be written as a product $f_j=hg$, with $h\neq 0$  non constant entire function and $g$ an entire function. Let $\V(\Li(V))$ be the set of vertices of $\Li(V)$. We choose a vertex $v\in \V(\Li(V))$ with slope $(u_1,\ldots,u_n)$. We denote by $D_v$ the straight line in $\R^n$ directed by $v$ and asymptotic to the amoeba $\mathscr{A}(V)$. Let $\mathscr{H}({D_v})$ be  the holomorphic cylinder  which is the lifting of $D_v$ by  $\Log_{|_V}$ and is asymptotic to the end of $V$ corresponding to $v$, such that $\Li(\mathscr{H}({D_v}))\cap \Li(V)=\{v\}$.
The functions $f_j$'s are entire and their power series expansions are of the form $\sum_{\alpha} c_{j\alpha} z^\alpha$  with $b_v+\sum_{i=1}^n u_i\alpha_i\leq 0$, where $b_v$ is a real number. In other words, the  exponents of the power series expansion of
 the $f_j$'s are contained in some half space depending on the slope of the vertex $v$.
By doing the same operation for all the vertices of $\Li(V)$ and using Lemma \ref{hen lemm} below, we conclude that  the exponents of the power series expansion of $f_j$ are contained in a compact polytope.
\end{proof}

 \begin{lemma}\label{hen lemm}
 Let $V$ be a subvariety in $(\C^*)^n$. If $S^{n-2}$ is a subsphere of $S^{n-1}=\partial B^n$, invariant under the involution $-\rm{id}$, then  
$\mathscr L^\infty (V)$ intersects the interior of each connected component of
 $S^{n-1}\setminus S^{n-2}$. 
\end{lemma}

\begin{proof}
Set $r:=n-\dim V-1$. We know that the complement components of amoebas are $r-$convex by Theorem \ref{henri}. Then the intersection of the closed half spaces in $\mathbb{R}^n$  bounded by the hyperplanes normal to all the directions $v\in \V({\Li}(V))$ is compact.  
\end{proof}

\vspace{0.2cm}

\begin{definition}\label{generic}
An analytic variety $V$ is generic, if $V$ is a finite union of irreducible components and each irreducible component of $V$ contains an open dense subset $U$, such that the Jacobian of the restriction of the logarithmic map to $U$ has maximal rank.
\end{definition}

We have the following proposition and corollaries:

\begin{proposition}\label{prop A}
Let $\mathcal{C}$ be a generic  analytic curve 
(not necessary algebraic) of $(\mathbb{C}^*)^n$. 
Then $\mathscr{L}^{\infty}(\mathcal{C})$ is the union of a finite number of isolated points with rational slopes and a finite number of  geodesic arcs with rational  end slopes. In particular, if $\mathcal{C}$ is not algebraic,  then the number of arcs in $\mathscr{L}^{\infty}(\mathcal{C})$ is different than zero.
\end{proposition}

\vspace{0.2cm}

\begin{proof}
It is sufficient to show that any point in $\mathscr{L}^{\infty}(\mathcal{C})$ with irrational slope is necessarily contained in the interior of $\mathscr{L}^{\infty}(\mathcal{C})$. 
 Without loss of generality, we can assume that the curve $\mathcal{C}$ is irreducible. We suppose on the contrary that there exists a point with irrational slope $s$, which is either isolated or is contained in the boundary  of a connected component of $\mathscr{L}^{\infty}(\mathcal{C})$. By Lemma 4.1 \cite{MN2-11}, the phase limit set $\mathscr P^\infty (\mathcal C)$ contains a subset of dimension at least two. More precisely, it contains an immersed circle $S$  of irrational  slope $s$ in the real torus $(S^1)^n$ such that its closure is at least $2$-dimensional. Let $U$ be an open subset of the torus $(S^1)^n$ such that $U\cap S$ is nonempty. Since the closure of the immersed circle $S$ is at least $2$-dimensional, then the intersection $U\cap S$ has an infinite number of connected components. For each such connected component $C_i$, we choose  an open subset $V_i$ of the regular part of the coamoeba such that $\partial \overline{V}_i$ contains $C_i$ and is of area a constant $A$ 
 %({\color{red}OR $A_i$?}) 
 different than zero where $\overline{V}_i$ denotes the closure of $V_i$ (i.e., the area of $V_i$ is equal to $A$ for all $i$). We claim that the union of the following subsets of the amoeba  $\tilde{V}_i :=(\Log_{|\mathcal{C}})\circ(\Arg_{|\mathcal{C}})^{-1}(V_i)$ is not bounded. Otherwise, if for every open set $U$ of the real torus this union is bounded, and by compactness of the torus, this implies that the amoeba $\mathscr{A}(\mathcal{C})$ has no tentacle of slope $s$. 
 
 Since the map $\Log_{|\mathcal C}\circ\Arg_{|\mathcal C}^{-1}$ conserves the area and $\mathcal{C}$ is generic, then for any positive number $R\gg 1$,  there exists an index $i$ such that the intersection of the sphere $S_{R}^{n-1}$ of radius $R$ with $\tilde{V}_i$ is of dimension one. Moreover,  the length of $I_i:= \tilde{V}_i\cap S_{R}^{n-1}$ does not  converge to zero. In fact,  using the convexity (or higher convexity in the case of higher codimension) of the amoeba complement, the intersection $I_i:= \tilde{V}_i\cap S_{R}^{n-1}$ must converge to a point by hypothesis. But in this case, the area of $\tilde{V}_i$ converges to zero too. This contradicts the fact that  for any index $i$ the area of   $\tilde{V}_i$ is equal to $A$. This implies that if a point $v$ in $\mathscr{L}^{\infty}(\mathcal{C})$ has an irrational slope, then $v$ must be in the interior of the logarithmic limit set.
 
\end{proof}

\vspace{0.2cm}

The phase limit set version of Proposition \ref{prop A} is the following:

\vspace{0.2cm}

%\newpage

\begin{corollary}\label{cor B}
Let $\mathcal{C}$ be a generic  analytic 
curve (not necessary algebraic) of $(\mathbb{C}^*)^n$. 
Then $\mathscr{P}^{\infty}(\mathcal{C})$ is an arrangement  of a finite number of geodesic circles with rational slopes and a finite number of  $2$-dimensional flat tori. In particular, if $\mathcal{C}$ is not algebraic,  the number of $2$-dimensional flat tori in $\mathscr{P}^{\infty}(\mathcal{C})$ is different than zero. Moreover, if $\mathcal{C}$ is not algebraic
and  $n=2$, then the closure of its coamoeba is the whole torus.
\end{corollary}

%\vspace{0.2cm}

\begin{proof} 
This corollary is a phase interpretation of Proposition \ref{prop A}. Indeed, if the curve is generic and  not algebraic, then the dimension of its logarithmic set is equal to one. This implies that the phase limit set contains an immersed circle with closure a  torus of dimension at least two.
\end{proof}

\begin{corollary}\label{cor C}
Let $V$ be a $k$-dimensional generic  analytic subvariety of the 
complex algebraic torus $(\mathbb{C}^*)^n$. % with $n\geq 2k$. 
If $V$ is not algebraic, then the closure of the coamoeba 
$\overline{co\mathscr{A}(V)}$ contains a flat torus of 
dimension at least $k+1$.
\end{corollary}

\begin{proof}
 More precisely, the coamoeba $co\mathscr{A}(V)$ of $V$ contains an immersed $k$-dimensional torus whose closure is of dimension at least $k+1$. In fact, if $V$ is a $k$-dimensional generic  analytic subvariety of the 
complex algebraic torus $(\mathbb{C}^*)^n$ and is not algebraic, then its logarithmic set is at least $k$-dimensional. In other words,
its phase limit  contains a torus of dimension at least $k+1$.
\end{proof}

\begin{figure}[ht!]
     \begin{center}
\subfigure[The amoeba   $\mathscr{A}$  of the analytic curve parametrized by $z\mapsto (z, e^z)$.
]{%
            \label{fig:first}
            \includegraphics[width=0.5\textwidth]{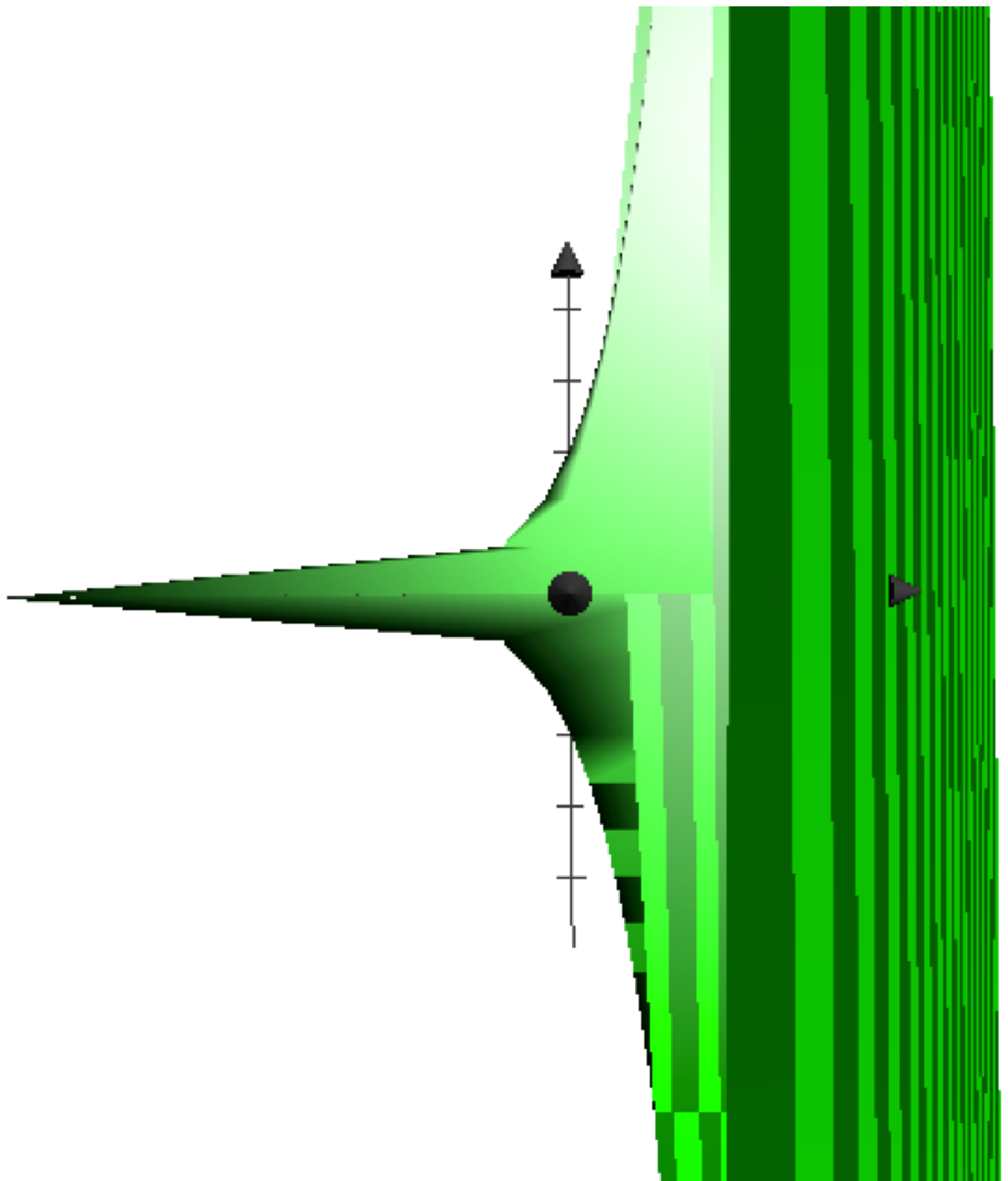}
        }%
        \subfigure[The image  of $\mathscr{A}$ by the retraction $\rho$.]{%
           \label{fig:second}
           \includegraphics[width=0.6\textwidth]{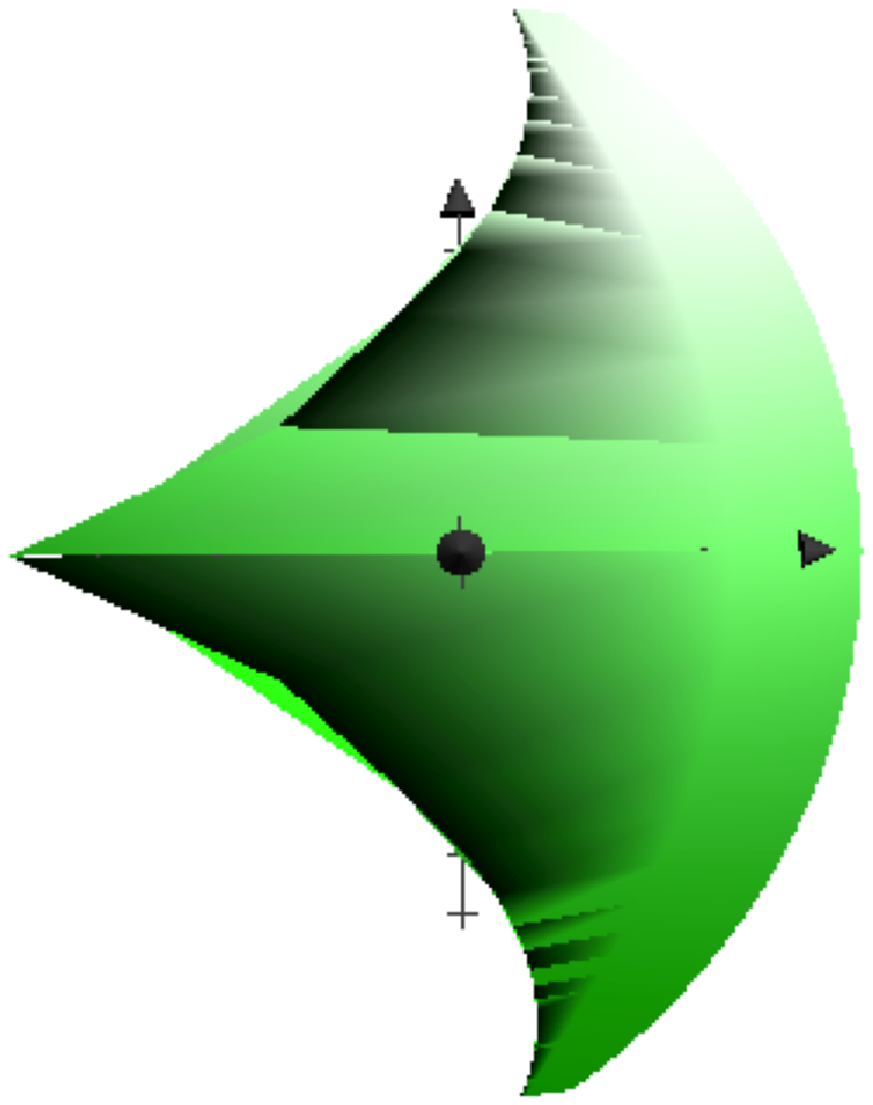}
        }\\

\end{center}
  %  \caption{%
      %  The amoeba of the analytic plane curve parametrized by $z\mapsto (z, e^z)$ and its image by $\rho$ in the disk.
 %    }%
   \label{fig:subfigures}
\end{figure}

\begin{figure}[ht!]
     \begin{center}
%\subfigure[The image by $\rho$ in the ball of the amoeba  of the  curve given by the parametrization $g(t)=(t,e^t,t+1)$.
%]{%
 %           \label{fig:first}
 %           \includegraphics[width=0.6\textwidth]{Holomor-Example-E.eps}
%        }%
        \subfigure[The image by $\rho$ in the ball of the amoeba  of the  curve given by the parametrization $g(t)=(t,e^t,t+1)$.]{%
           \label{fig:second}
           \includegraphics[width=1\textwidth]{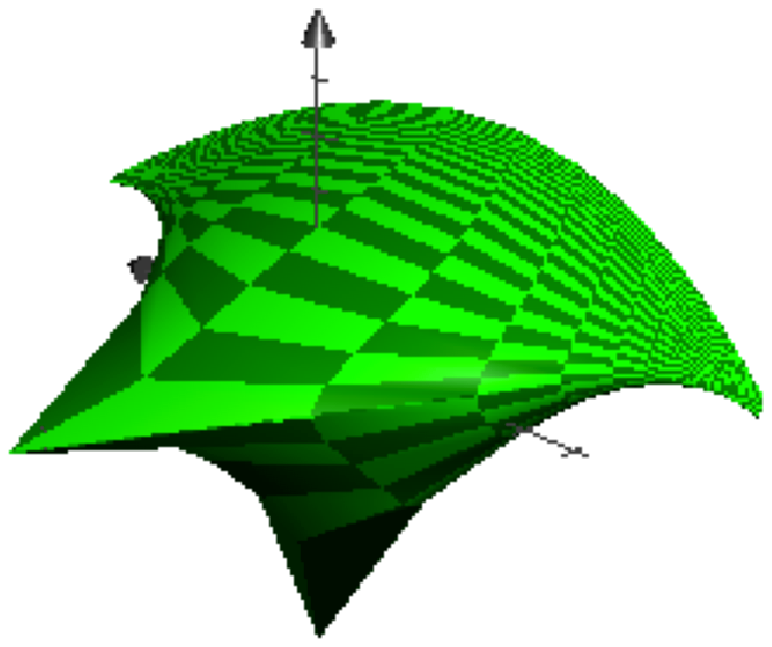}
        }\\

\end{center}
 %   \caption{%
 %       The image by $\rho$ in the ball of the amoeba  of the  curve given by the parametrization $g(t)=(t,e^t,t+1)$
 %  of the analytic plane curve parametrized by $z\mapsto (z, e^z)$ and its image by $\rho$ in the disk.
%     }%
   \label{fig:subfigures}
\end{figure}

%
%
%
%
%

%\bibliographystyle{amsalpha}
%\bibliography{biblioM}

\end{document}